\documentclass[10pt]{article}
\usepackage{amssymb}
\usepackage{amsmath}
\usepackage{latexsym}
\newtheorem{theorem}{Theorem}[section]
\newtheorem{definition}{Definition}[section]
\newtheorem{lemma}{Lemma}[section]
\newtheorem{proposition}{Proposition}[section]
\newtheorem{remark}{Remark}[section]

\newtheorem{corollary}{Corollary}[section]
\newenvironment{proof}[1][Proof]{\noindent\textbf{#1.} }{\ \rule{0.5em}{0.5em}}
\makeatletter\def\theequation{\arabic{section}.\arabic{equation}}\makeatother
\begin{document}
\date{}
\title{
{\bf\Large  Torsion functions and the Cheeger problem: a fractional approach}}
\author{{\bf\large H. Bueno, G. Ercole, S.S. Macedo  and
G.A. Pereira}\footnote{The authors acknowledge the support of CNPq-Brazil and FAPEMIG.}\hspace{2mm}
\vspace{1mm}\\
{\it\small Departamento de Matem\'{a}tica}\\ {\it\small Universidade Federal de Minas Gerais},
{\it\small Belo Horizonte, 31.270-901, Brazil}\\
{\it\small e-mail: hamilton@mat.ufmg.br, grey@mat.ufmg.br, sdasilvamacedo@gmail.com}\\ 
{\it\small and gilbertoapereira@yahoo.com.br}}

\maketitle
\noindent
{\bf{Abstract}}: {{\small Let $\Omega$ be a Lipschitz bounded domain of $\mathbb{R}^N $, $N\geq2$. The fractional Cheeger constant $h_s (\Omega)$, $0<s<1$, is defined by $h_s(\Omega)=\inf_{E\subset{\Omega}}\frac{P_s(E)}{|E|}$, where $P_s (E)=\int_{\mathbb{R}^N }\int_{\mathbb{R}^N }\frac{|\chi_{E}(x)-\chi_{E}(y)|}{|x-y|^{N+s}}\textup{d}x\textup{d}y$ with $\chi_{E}$ denoting the characteristic function of the smooth subdomain $E$. The main purpose of this paper is to show that $\lim_{p\rightarrow1^+}\left|\phi_p^s\right|_{L^{\infty}(\Omega)}^{1-p}=h_s (\Omega)=\lim_{p\rightarrow1^+}\left|\phi_p^s\right|_{L^1(\Omega)}^{1-p}$, where $\phi_p^s$ is the fractional $(s,p)$-torsion function of $\Omega$, that is, the solution of the Dirichlet problem for the fractional $p$-Laplacian: $-(\Delta)_p^s\,u=1$ in $\Omega$, $u=0$ in $\mathbb{R}^N \setminus\Omega$. For this, we derive suitable bounds for the first eigenvalue $\lambda_{1,p}^s(\Omega)$ of the fractional $p$-Laplacian operator in terms of $\phi_p^s$.  We also show that $\phi_p^s$ minimizes the $(s,p)$-Gagliardo seminorm in $\mathbb{R}^N $, among the functions normalized by the $L^1$-norm.}}
\vspace{3mm}

\noindent
	{\bf{Key words}}: {\small Fractional Cheeger problem, fractional $p$-Laplacian, fractional torsion functions.}

\noindent
{\bf {MSC 2010}}: {\small Primary 35P15, 35R11; Secondary 47A75}.\\

\section{\bf Introduction}
\def\theequation{1.\arabic{equation}}\makeatother
\setcounter{equation}{0}

The Cheeger constant $h(\Omega)$ of a bounded domain $\Omega\subset
\mathbb{R}^N $ ($N>1$) is defined by
\begin{equation}
h(\Omega)=\inf_{E\subset\Omega}\frac{P(E)}{|E|}, \label{Cheeger}%
\end{equation}
where $E$ is a smooth subset of $\Omega$ and the nonnegative values $P(E)$ and
$|E|$ denote, respectively, the distributional perimeter and the
$N$-dimensional Lebesgue measure of $E$. A subset $E$ that minimizes the
quotient is a Cheeger set of $\Omega$.

In \cite{Kawohl1} Kawohl and Fridman proved that
\[
h(\Omega)=\lim_{p\rightarrow1^+ }{\lambda_{1,p}}(\Omega),
\]
where $\lambda_{1,p}(\Omega)$ is the first eigenvalue of the Dirichlet
$p$-Laplacian operator, that is, the least real number $\lambda$ such that the
Dirichlet problem%
\[
\left\{
\begin{array}
[c]{rcll}%
-\Delta_p    u&=&\lambda\left\vert u\right\vert ^{p-2}u & \mathrm{in }\  \Omega\\
u&=&0 & \mathrm{on }\  \partial\Omega,
\end{array}
\right.
\]
has a nontrivial solution. (Let us recall that the $p$-Laplacian operator is
defined by $\Delta_p    u:=\textup{div}\,(|\nabla u|^{p-2}\nabla u)$, $p>1$. )

The first eigenvalue ${\lambda_{1,p}}(\Omega)$ is also variationally characterized by
\[\lambda_{1,p}(\Omega):=\min\left\{\left|\nabla u\right|_p^p\,:\, u\in W_0 ^{1,p}(\Omega),\left|u\right|_p=1\right\}  ,
\]
with $\left|\cdot\right|_r$ standing for the usual norm of $L^r(\Omega)$, $1\leq r\leq\infty$ (this notation will be adopted from now on).

In \cite{Bueno} a different characterization of the Cheeger constant of $\Omega$ was obtained:
\[\lim_{p\rightarrow1^+ }\frac{1}{|\phi_p|_\infty^{p-1}}
=h(\Omega)=\lim_{p\rightarrow1^+ }\frac{1}{|\phi_p|_1^{p-1}}
\]
where $\phi_p$ denotes the $p$-torsion function of $\Omega$, that is, the solution of the Dirichlet problem
\[\left\{\begin{array}[c]{rcll}
-\Delta_p u&=&1 & \mathrm{in }\ \Omega\\
u&=&0 & \mathrm{on }\ \partial\Omega,
\end{array}\right.
\]
which is known as the $p$-torsional creep problem (see \cite{Creep}).

We remark that
\[\frac{1}{|\phi_p|_1^{p-1}}=\min\left\{\left|\nabla
u\right|_p^p\,:\, u\in W_0^{1,p}(\mathbb{R}^N),\ \left|u\right|_1=1\right\}  ,
\]
since the minimum is attained at $\phi_p/\left|\phi_p\right|_1$.

The fractional version of problem (\ref{Cheeger}) consists in minimizing that quotient when $P(E)$ is substituted by $P_s(E)$, given by
\[P_s(E)=\int_{\mathbb{R}^N}\int_{\mathbb{R}^N}\frac{|\chi_E(x)-\chi_E(y)|}{|x-y|^{N+s}}\mathrm{d}x\mathrm{d}y,
\]
where $\chi_E$ stands for the characteristic function of the smooth subdomain $E$. The value $P_s(E)$ is called the (nonlocal) $s$-perimeter of $E$. So, the fractional Cheeger problem is the minimization problem
\[h_s(\Omega)=\inf_{E\subset{\Omega}}\frac{P_s(E)}{|E|},
\]
and $h_s(\Omega)$ is called the $s$-Cheeger constant of $\Omega$.

For $1<p<\infty$ and $s\in(0,1)$, the fractional $(s,p)$-Laplacian $(-\Delta)_p^s$ is the nonlinear nonlocal operator defined by
\[(-\Delta)_p^s\,u(x)=\lim_{\epsilon\rightarrow0^+ }\int_{\mathbb{R}^N\setminus B_\epsilon(x)}\frac{|u(x)-u(y)|^{p-2}(u(x)-u(y))}{|x-y|^{N+sp}}\mathrm{d}y.
\]
This definition is consistent, up to a normalization constant depending on
$p$, $s$ and $N$, with the usual $p$-Laplacian operator.

The first (fractional) eigenvalue $\lambda_{1,p}^s(\Omega)$ of
$(-\Delta)_p^s$ is the least number $\lambda$ such that the problem
\[\left\{\begin{array}[c]{rcll}%
(-\Delta)_p    ^s  \,u&=&\lambda\left|u\right|^{p-2}u & \mathrm{in }\ \Omega\\
u&=&0 & \mathrm{on }\  \mathbb{R}^N\setminus\Omega,
\end{array}\right.
\]
has a nontrivial weak solution (see \cite{Erilind,FP}). Its variational characterization is given by (see \cite{BLP}):
\begin{equation}
\lambda_{1,p}^s(\Omega):=\min\left\{\left[u\right]  _{s,p}^p\,:\,u\in W_0^{s,p}(\Omega),\ \left|u\right|_p=1\right\},\label{minLp}%
\end{equation}
where
\begin{equation}
\left[u\right]_{s,p}:=\left(  \int_{\mathbb{R}^N}\int_{\mathbb{R}^N}\frac{|u(x)-u(y)|^p}{|x-y|^{N+sp}}\mathrm{d}x\mathrm{d}y\right)^{\frac
{1}{p}}\label{Gag}%
\end{equation}
is the $(s,p)$-seminorm of Gagliardo in $\mathbb{R}^N $ of a measurable function $u$ and $W_0^{s,p}(\Omega)$ is a suitable fractional Sobolev space defined in the sequel (see Definition \ref{wtil}).

In \cite{BLP} Brasco, Lindgren and Parini proved the $s$-Cheeger version of the result originally obtained by Kawohl and Friedman \cite{Kawohl1} for the Cheeger problem:
\begin{equation}
h_s(\Omega)=\lim_{p\rightarrow1^+ }{\lambda_{1,p}^s}(\Omega).\label{skaw}%
\end{equation}

In this paper, by assuming that $\Omega$ is a Lipschitz bounded domain, we show, in the spirit of the paper \cite{Bueno}, that the fractional version of the torsional creep problem,%
\begin{equation*}
\left\{\begin{array}[c]{rcll}%
(-\Delta)_p    ^s  \,u&=&1 & \mathrm{in }\ \Omega\\
u&=&0 & \mathrm{on }\ \mathbb{R}^N\setminus\Omega,
\end{array}\right.  \label{stor}%
\end{equation*}
is intrinsically connected to both the $s$-Cheeger problem and the first eigenproblem for the fractional Dirichlet $p$-Laplacian, as $p$ goes to $1$.

This connection will be developed in Section \ref{sec2}, where we introduce the $(s,p)$-torsion function of $\Omega$, that is, the weak solution
$\phi_p^s$ of (\ref{stor}). We will derive the  estimates
\begin{equation}
\frac{1}{\left|\phi_p^s\right|_\infty^{p-1}}\leq\lambda
_{1,p}^s(\Omega)\leq\left(\frac{|\Omega|}{\left|\phi_p^s\right|_1}\right)^{p-1}\label{bounds}%
\end{equation}
and
\begin{equation}
\frac{\left|\phi_p^s\right|_\infty}{\left|\phi_p^s\right|_1}\leq\frac{1}{\left|B_1\right|}\left(
\frac{sp+N(p-1)}{sp}\right) ^{\frac{sp+N(p-1)}{sp}}\left(\frac
{\lambda_{1,p}^s(\Omega)}{\lambda_{1,p}^s(B_1)}\right) ^{\frac{N}{sp}},\label{estim}%
\end{equation}
where $B_1$ denotes the unit ball of $\mathbb{R}^N$.

Then, taking \eqref{skaw} into account, we will combine \eqref{bounds} with \eqref{estim} in order to conclude the main result of this paper:
\begin{equation*}
\lim_{p\rightarrow1^+ }\frac{1}{\left|\phi_p^s\right|_\infty^{p-1}}=h_s(\Omega)=\lim_{p\rightarrow1^+}\frac{1}{\left|\phi_p^s\right|_1^{p-1}}.\label{sptorsionsCheeger}
\end{equation*}

Still in Section \ref{sec2} we prove that $\phi_p^s$ minimizes, in $W_0^{s,p}(\Omega)\setminus\left\{  0\right\}$, the Rayleigh quotient
$\left[u\right]_{s,p}^p/\left|u\right|_1^p$. As an
immediate consequence of this fact, we show that $\phi_p^s$ is a radial function when $\Omega$ is a ball.

\section{\bf The main results}\label{sec2}
\def\theequation{2.\arabic{equation}}
\makeatother
\setcounter{equation}{0}

From now on $\Omega$ denotes a Lipschitz bounded domain of $\mathbb{R}^N$, $N\geq2$, and $0<s<1< p<\frac{N}{s}$. 

\begin{definition}\label{wtil} The Sobolev space $W_0^{s,p}(\Omega)$ is the closure of $C_0^{\infty}(\Omega)$ with respect to the norm
\begin{equation}
\left\|u\right\|:=\left[u\right]_{s,p}+\left|u\right|
_p,\label{norma2}%
\end{equation}
where $\left[u\right]_{s,p}$ is defined by \eqref{Gag}.
\end{definition}

Functions in $W_0^{s,p}(\Omega)$ have a natural extension to $\mathbb{R}^N$ and, although $u=0$ in $\mathbb{R}^N\setminus\Omega$, the identity
\[\left[u\right] _{s,p}^p=\int_\Omega \int_\Omega\frac{|u(x)-u(y)|^p}{|x-y|^{N+sp}}\mathrm{d}x\mathrm{d}y+2\int_{\mathbb{R}^N\setminus\Omega}\int_\Omega\frac{|u(x)|^p}{|x-y|^{N+sp}}\mathrm{d}x\mathrm{d}y
\]
shows dependence on values in $\mathbb{R}^N\setminus\Omega$.

It is worth mentioning that $W_0 ^{s,p}(\Omega)$ is a reflexive Banach space and that this space coincides with the closure of $C_0^\infty(\Omega)$ relative to the norm
\[u\longmapsto\left(  \int_\Omega\int_\Omega\frac{|u(x)-u(y)|^p}{|x-y|^{N+sp}}\mathrm{d}x\mathrm{d}y\right) ^{\frac{1}{p}}+\left|u\right|_p,
\]
if $\partial\Omega$ is Lipschitz, (see \cite[Proposition B.1]{BLP}).

Moreover, thanks the fractional Poincar\'{e} inequality (see \cite[Lemma 2.4]{BLP})
\[\left|u\right|_p^p\leq C_{N,s,p,\Omega}\left[  u\right]_{s,p}^p,\quad\forall\,u\in C_0^\infty(\Omega),
\]
$\left[\cdot\right]_{s,p}$ is also a norm in $W_0^{s,p}(\Omega)$, equivalent to norm defined in \eqref{norma2}.

We refer the reader to \cite{Guide} for fractional Sobolev spaces.

\begin{definition}
We say that a function $u\in W_0^{s,p}(\Omega)$ is a weak solution of the fractional Dirichlet problem%
\begin{equation}
\left\{\begin{array}[c]{rcll}%
(-\Delta)_p^s\,u&=&f & \mathrm{in }\ \Omega\\
u&=&0 &\mathrm{on }\ \mathbb{R}^N\setminus\Omega,
\end{array}\right.  \label{fdiric}%
\end{equation}
if 
\begin{equation}
\left\langle(-\Delta)_p^s\,u,\varphi\right\rangle=
\int_\Omega
f\varphi\,\mathrm{d}x,\quad\forall\,\varphi\in W_0^{s,p}(\Omega),\label{fweak}%
\end{equation}
where
\[\left\langle (-\Delta)_p^s\,u,\varphi\right\rangle :=\int_{\mathbb{R}^N}\int_{\mathbb{R}^N}\frac{|u(x)-u(y)|^{p-2}(u(x)-u(y))(\varphi(x)-\varphi(y))}{|x-y|^{N+sp}}\mathrm{d}x\mathrm{d}y,
\]
a notation that will be used from now on.
\end{definition}

Existence of a weak solution of \eqref{fdiric}, when $f\in L^1(\Omega)$ follows from direct minimization in $W_0^{s,p}(\Omega)$ of the functional
\[\frac{1}{p}\left[u\right]_{s,p}^p-\int_\Omega f(x)u(x)\,\mathrm{d}x,
\]
whereas uniqueness comes from, for instance, the comparison principle for the fractional $p$-Laplacian (see \cite[Lemma 9]{Erilind}). The same principle shows that if $f$ is nonnegative then the weak solution $u$ is nonnegative as well. When $f\in L^\infty(\Omega)$ and $\Omega$ is sufficiently smooth, say with boundary at least of class $C^{1,1}$, the weak solutions are $\alpha$-H\"{o}lder continuous up to the boundary for some $\alpha\in(0,1)$, see \cite{IMS}. 

When $f\equiv1$ the Dirichlet problem \eqref{fdiric} will be referred to as the $(s,p)$-fractional torsional creep problem and its unique weak solution will be called $(s,p)$-torsion function. Let us denote this function by $\phi_p^s$. We have $\phi_p^s\geq0$ and
\begin{equation}
\left\langle(-\Delta)_p^s\,\phi_p^s,\varphi\right\rangle =\int_\Omega \varphi\,\mathrm{d}x,\quad\forall\,\varphi\in W_0^{s,p}(\Omega).\label{weaktor}%
\end{equation}
In particular, by taking $\varphi=\phi_p^s$ we obtain
\[\left\langle(-\Delta)_p^s\,\phi_p^s,\phi_p^s\right\rangle=\left[\phi_p^s\right]_{s,p}^p=\left|\phi_p^s\right|_1.
\]

\begin{theorem}
We have
\begin{equation}
\frac{1}{\left|\phi_p^s\right|_1^{p-1}}=\min\left\{  \left[v\right]_{s,p}^p\,:\,v\in W_0^{s,p}(\Omega),\, \left|v\right|_1=1\right\} =\left[ \frac{\phi_p^s}{\left|\phi_p^s\right|_1}\right] _{s,p}.\label{minL1}%
\end{equation}
Moreover, $\frac{\phi_p^s}{\left|\phi_p^s\right|_1}$ is the only nonnegative function attaining the minimum.
\end{theorem}
\begin{proof}
Since the functional $v\longmapsto\left|v\right|_1$ is not differentiable, we will first consider the minimization problem
\begin{equation}
m:=\inf\left\{\left[v\right]_{s,p}^p\,:\,v\in W_0^{s,p}(\Omega),\,\int_\Omega
v\,\mathrm{d}x=1\right\} \label{m}%
\end{equation}
and show that it is uniquely solved by the positive function $\frac{\phi_p^s}{\left|\phi_p^s\right|_1 }$. 

Thus, let us take a sequence $\left(u_n \right) \subset W_0 ^{s,p}(\Omega)$ such that
\[\int_\Omega u_n\, \mathrm{d}x=1\quad\mathrm{and}\quad\left[u_n \right]_{s,p}^p\rightarrow m.
\]

We observe that the sequence $\left(u_n \right)$ is bounded in $L^1(\Omega)$:
\[\left|u_n\right|_1^p\leq\left|\Omega\right|^{p-1}\left|u_n\right|_p^p\leq\left|\Omega\right|^{p-1}\lambda_{1,p}^s(\Omega)^{-1}\left[u_n\right]_{s,p}^p.
\]

Since $W_0 ^{s,p}(\Omega)$ is reflexive, the $L^1$-boundedness of $\left(u_n\right)$ implies the  existence of $u\in W_0^{s,p}(\Omega)$ such
that, up to a subsequence, $u_n\rightharpoonup u$ (weak convergence) in $W_0^{s,p}(\Omega)$ and $u_n \rightarrow u$ in $L^1(\Omega)$. The convergence in $L^1(\Omega)$ implies that ${\displaystyle\int_\Omega} u\,\mathrm{d}x=1$, so that $m\leq\left[u\right] _{s,p}^p$. On its turn, the weak convergence guarantees that
\[\left[u\right]_{s,p}+1=\left\|u\right\|\leq\liminf\left\|u_n\right\|=\liminf\left(\left[u_n\right]  _{s,p}+1\right)=m^{\frac{1}{p}}+1.
\]
It follows that $m=\left[u\right]_{s,p}^p$, so that the infimum in \eqref{m} is attained by the weak limit $u$.

By applying Lagrange multipliers, we infer the existence of a real number $\lambda$ such that
\begin{equation}
\left\langle(-\Delta)_p^s\,u,\varphi\right\rangle =\lambda\int_\Omega\varphi\,\mathrm{d}x,\quad\forall\,\varphi\in W_0 ^{s,p}(\Omega).\label{lag}%
\end{equation}
Taking $\varphi=u$ we conclude that $\lambda=m>0$, since $m=\left[u\right]_{s,p}^p$ and ${\displaystyle\int_\Omega}u\,\mathrm{d}x=1$. 

This fact and \eqref{lag} imply that $u$ is a weak solution of the Dirichlet problem
\[\left\{\begin{array}[c]{rcll}
(-\Delta)_p^s\,u&=&m & \mathrm{in }\ \Omega\\
u&=&0 & \mathrm{on }\ \mathbb{R}^N\setminus\Omega.
\end{array}\right.
\]

By uniqueness, we have $u=m^{\frac{1}{p-1}}\phi_p^s\geq0$. Since ${\displaystyle\int_\Omega}u\,\mathrm{d}x=1$ we conclude that
\[m=\frac{1}{\left|\phi_p^s\right|_1^{p-1}}\quad\mathrm{and}\quad u=\frac{\phi_p^s}{\left|\phi_p^s\right|_1}.
\]

We remark that
\[\frac{1}{\left|\phi_p^s\right|_1 ^{p-1}}\leq\left[\,\left|v\right|\,\right]_{s,p}^p \leq\left[v\right]_{s,p}^p
\]
for every $v\in W_0 ^{s,p}(\Omega)$ such that $\left|v\right|_1=1$. This finishes the proof since
\[\frac{1}{\left|\phi_p^s\right|_1 ^{p-1}}=\left[  \frac{\phi_p^s}{\left|\phi_p^s\right|_1 }\right]_{s,p}\quad\mathrm{and}\quad\left|\frac{\phi_p^s}{\left|\phi_p^s\right|_1 }\right|_1 =1.
\]
\end{proof}\goodbreak

The next result recovers Lemma 4.1 of \cite{IMS}:
\begin{corollary}
The $(s,p)$-torsion function is radial when $\Omega$ is a ball.
\end{corollary}
\begin{proof}
Let $(\phi_p^s)^{\ast}\in W_0 ^{s,p}(\Omega)$ be the Schwarz symmetrization of $\phi_p^s$, that is, the radially decreasing function such that
\[
\left\{\phi_p^s>t\right\}^*=\left\{
(\phi_p^s)^*>t\right\},\ t>0,
\]
where, for any $D\subset\mathbb{R}^N$, $D^*$ stands for the $N$-dimensional ball with the same volume of $D$. 

It is well-known that $(\phi_p^s)^*\geq0$, $\left[
(\phi_p^s)^*\right]_{s,p}\leq\left[  \phi_p^s\right]_{s,p}$
and $\left|(\phi_p^s)^*\right|_1=\left|\phi_p^s  \right|_1$.  Therefore, $(\phi_p^s)^*/\left|(\phi_p^s)^*\right|_1$ attains the minimum in \eqref{minL1} and by uniqueness we have $(\phi_p^s)^*=\phi_p^s$. 
\end{proof}\vspace*{.5cm}

It is also well-known that the first eigenfunctions of the fractional $p$-Laplacian belong to $L^\infty(\Omega)$ and are either positive or negative almost everywhere in $\Omega$. Moreover, they are scalar multiple each other. So, let us denote by $e_p^s$ the positive and $L^\infty$-normalized first eigenfunction. It follows that $\left|e_p^s\right|_\infty=1$
and
\[\left\{\begin{array}[c]{rcll}
(-\Delta)_p^s\,e_p^s&=&\lambda_{1,p}^s  (\Omega)(e_p^s)^{p-1} & \mathrm{in }\ \Omega\\
e_p^s&=&0 & \mathrm{on }\ \mathbb{R}^N\setminus\Omega,
\end{array}\right.
\]
meaning that%
\begin{equation}
\left\langle (-\Delta)_p^s\,e_p^s ,\varphi\right\rangle =\lambda_{1,p}^s(\Omega)\int_\Omega (e_p^s)^{p-1}\varphi\,\mathrm{d}x,\quad\forall\,\varphi\in W_0^{s,p}(\Omega).\label{weakeig}
\end{equation}
Of course, by taking $\varphi=e_p^s$ in \eqref{weakeig} we obtain%
\[\left\langle(-\Delta)_p^s\,e_p^s,e_p^s \right\rangle =\left[e_p^s\right]_{s,p}^p  =\lambda_{1,p}^s(\Omega)\left|e_p^s\right|_p^p.
\]

As mentioned in the introduction, $\lambda_{1,p}^s  (\Omega)$ is variationally
characterized by \eqref{minLp}.

\begin{proposition}
Let $u\in W_0^{s,p}(\Omega)$ be the weak solution of \eqref{fdiric} with $f\in L^\infty(\Omega)\setminus\left\{0\right\}$. Then,
\begin{equation}
\left|f\right|_\infty^{-\frac{1}{p-1}}u\leq\phi_p^s\quad\textrm{a.e. in }\Omega\label{compf}
\end{equation}
and
\begin{equation}
\lambda_{1,p}^{s}(\Omega)\leq\left|f\right|_\infty\left(
\frac{\left|\Omega\right|}{\left|u\right|_1 }\right)
^{p-1}. \label{upperf}%
\end{equation}
\end{proposition}
\begin{proof}
Since $u$ and $\phi_p^s$ are both equal zero in $\mathbb{R}^N\setminus\Omega$ and
\begin{align*}
\left\langle (-\Delta)_p^s  \,\left(\left|f\right|_\infty^{-\frac{1}{p-1}}u\right),\varphi\right\rangle&=\left|f\right|_\infty^{-1}\left\langle(-\Delta)_p^s\,u,\varphi\right\rangle\\
&=\left|f\right|_\infty^{-1}
\int_\Omega f\varphi\,\mathrm{d}x\\
&\leq\int_{f\geq0}\varphi\,\mathrm{d}x\leq\int_\Omega
\varphi\,\mathrm{d}x\leq\left\langle(-\Delta)_p^s  \,\phi_p^s,\varphi\right\rangle
\end{align*}
holds for every nonnegative $\varphi\in W_0 ^{s,p}(\Omega)$,  \eqref{compf} follows from the
comparison principle (see \cite[Lemma 9]{Erilind}).

In order to prove \eqref{upperf} we use \eqref{minLp} and H\"{o}lder inequality:
\begin{align*}
\lambda_{1,p}^{s}(\Omega)&\leq\frac{\left[u\right]  _{s,p}^p}{\left|u\right|_p^p}\\
&=\frac{{\displaystyle\int_\Omega}
fu\,\mathrm{d}x}{\left|u\right|_p^p}\leq\frac{\left|
f\right|_\infty\left|u\right|_1}{\left|u\right|_p^p}\leq\frac{\left|f\right|_\infty\left|u\right|_1 }{\left|u\right|_1^p}\left|\Omega\right|^{p-1}=\left\|f\right|_\infty\left(\frac{\left|\Omega\right|}{\left|u\right|_1^p}\right)^{p-1}.
\end{align*}

\end{proof}

\begin{corollary}
It holds
\begin{equation}
e_p^s\leq\lambda_{1,p}^{s}(\Omega)^{\frac{1}{p-1}}\phi_p^s\quad\textrm{a.e. in }\Omega\label{comp}%
\end{equation}
and
\begin{equation}
\frac{1}{\left|\phi_p^s\right|_\infty^{p-1}}\leq\lambda_{1,p}^{s}(\Omega)\leq\frac{\left|\Omega\right|^{p-1}}{\left|\phi_p^s\right|_1 ^{p-1}}. \label{eigenbounds}%
\end{equation}
\end{corollary}
\begin{proof}
Taking $u=e_p^s$ and $f=\lambda_{1,p}(\Omega)(e_p^s )^{p-1}$ in \eqref{compf} we readily obtain \eqref{comp}. Hence, passing to maxima, we arrive at the first inequality in \eqref{eigenbounds}. The second inequality, on its turn, follows from \eqref{upperf} with $u=\phi_p^s$ and $f\equiv1$. 
\end{proof}\vspace*{.5cm}

We would like to emphasize the following consequence of \eqref{comp}: $\phi_p^s>0$ almost everywhere in $\Omega$.  

A Faber-Krahn inequality also holds true for the first fractional eigenvalue.

\begin{lemma}
[Theorem 3.5 of \cite{BLP}]\label{FK}Let $p>1$ and $s\in(0,1)$.  For every
bounded domain $D\subset\mathbb{R}^N $ we have
\begin{equation}
\left|B_1\right|^{\frac{sp}{N}}\lambda_{1,p}^s(B_1)
=\left|B\right|^{\frac{sp}{N}}\lambda_{1,p}^s  (B)\leq\left|D\right|^{\frac{sp}{N}}\lambda_{1,p}^s (D)\label{FaKr}%
\end{equation}
where $B$ is any $N$-dimensional ball and $B_1$ denotes the unit ball of $\mathbb{R}^N$. 
\end{lemma}

\begin{remark}
Since $h_s (D)=\lim_{p\rightarrow1^+ }\lambda_{1,p}^s(D)$ one has,
immediately,
\begin{equation}
\left|B_1 \right|^{\frac{s}{N}}h_s(B_1)=\left|B\right|
^{\frac{s}{N}}h_s(B)\leq\left|D\right|^{\frac{s}{N}}h_s(D).\label{FK2}%
\end{equation}
\end{remark}

The next estimate is obtained by applying standard set-level techniques; however, the bounds obtained are adequate to study the asymptotic behavior as $p\to 1^+$.
\begin{proposition}\label{fbounds}
Let $u\in W_0^{1,p}(\Omega)\setminus\left\{0\right\}$ be
a nonnegative, weak solution of \eqref{fdiric} with $f\in L^\infty(\Omega)$. 
Then $u\in L^\infty(\Omega)$ and
\begin{equation}
\frac{\left|u\right|_\infty}{\left|u\right|_1}
\leq\frac{1}{\left|B_1\right|}\left( \frac{sp+N(p-1)}{sp}\right)^{\frac{sp+N(p-1)}{sp}}\left(\frac{\left|f\right|_\infty}{\lambda_{1,p}^s  (B_1)\left|u\right|_\infty^{p-1}}\right)^{\frac{N}{sp}}.\label{estimf}%
\end{equation}
\end{proposition}
\begin{proof}
For each $k>0$ we set
\[A_k=\left\{x\in\Omega\,:\,u(x)>k\right\}.
\]
Since $u\in W_0 ^{s,p}(\Omega)$ and $u\geq0$ in $\Omega$, the function
\[(u-k)^+ =\max\left\{u-k,0\right\}= \left\{\begin{array}[c]{ccc}%
u-k, & \mathrm{if} & u>k\\
0, & \mathrm{if} & u\leq k
\end{array}\right.
\]
belongs to $W_0 ^{s,p}(\Omega)$. Therefore, choosing $\varphi=(u-k)^+$ in
\eqref{fweak} we obtain
\[\left\langle(-\Delta)_p^s\,u,(u-k)^+\right\rangle =\int_{A_k}f(x)(u-k)\,\mathrm{d}x.
\]

It is not difficult to check that
\[[(u-k)^+]_{s,p}^{p}\leq\left\langle(-\Delta)_p^s  \,u,(u-k)^+\right\rangle.
\]
Thus, we have
\begin{equation}
[(u-k)^+]_{s,p}^{p}\leq\int_{A_k}(u-k)f\,\mathrm{d}x\leq\left|f\right|_\infty\int_{A_k}(u-k)\,\mathrm{d}x. \label{desigvmais}%
\end{equation}

We now consider $k>0$ such that $\left|A_k\right|>0$. In order to
estimate $[(u-k)^+]_{s,p}^{p}$ from below, let us fix a ball $B\subset\mathbb{R}^N$ and apply Lemma \ref{FK} to obtain
\[\left|B_1\right|^{\frac{sp}{N}}\lambda_{1,p}^{s}(B_1)
\left|A_k\right|^{-\frac{sp}{N}}\leq\lambda_{1,p}^s  (A_k)\leq\frac{[(u-k)^+]_{s,p}^p}{\int_{A_k}(u-k)^p\,\mathrm{d}x}.
\]
Hence, H\"{o}lder's inequality yields
\[\left(\int_{A_k}(u-k)\,\mathrm{d}x\right)^p   \leq|A_k|^{p-1}\int_{A_k}(u-k)^p\,   \mathrm{d}x\leq|A_k|^{p-1}\frac{|A_k|^{\frac{sp}{N}}
[(u-k)^+]_{s,p}^p}{\left|B_1\right|^{\frac{sp}{N}}
\lambda_{1,p}^s(B_1)}.
\]

Thus, it follows from \eqref{desigvmais} that
\begin{align*}
\left|B_1\right|^{\frac{sp}{N}}\lambda_{1,p}^s(B_1)
|A_k|^{-\frac{sp+N(p-1)}{N}}\left(\int_{A_k} (u-k)\,\mathrm{d}x\right)^p&\leq[(u-k)^+]_{s,p}^p\\
&\leq\left|f\right|_\infty\int_{A_k}(u-k)\,\mathrm{d}x,
\end{align*}
what yields
\[\left(\int_{A_k}(u-k)\,\mathrm{d}x\right) ^{p-1}\leq\frac{\left|f\right|_\infty|A_k|^{\frac{sp+N(p-1)}{N}}}{\left|B_1\right|^{\frac{sp}{N}}\lambda_{1,p}^s(B_1)}
\]
and so
\begin{equation}
\left(\int_{A_k}(u-k)\,\mathrm{d}x\right) ^{\frac{N(p-1)}{sp+N(p-1)}}
\leq\left(\frac{\left|f\right|_\infty}{\left|B_1\right|^{\frac{sp}{N}}\lambda_{1,p}^s(B_1)}\right)^{\frac
{N}{sp+N(p-1)}}|A_k|.\label{des1}%
\end{equation}

Define
\[g(k):=\int_{A_k}(u-k)\,\mathrm{d}x=\int_k ^{\infty}|A_k|\,\mathrm{d}t,
\]
the last equality being a consequence of Cavalieri's principle. Combining the definition of $g(k)$ with \eqref{des1}, we have
\[[g(k)]^{\frac{N(p-1)}{sp+N(p-1)}}\leq-\left(  \frac{\left|f\right|_\infty}{\left|B_1\right|^{\frac{sp}{N}}\lambda_{1,p}^s(B_1)}\right)  ^{\frac{N}{sp+N(p-1)}}g^{\prime}(k).
\]

Therefore,
\begin{equation}
1\leq-\left(\frac{\left|f\right|_\infty}{\left|B_1 \right|^{\frac{sp}{N}}\lambda_{1,p}^s(B_1)}\right)^{\frac{N}{sp+N(p-1)}}\left[g(k)\right]  ^{-\frac{N(p-1)}{sp+N(p-1)}}g'(k).
\label{paraintegral}%
\end{equation}

Integration of \eqref{paraintegral} from $0$ to $k$ produces
\begin{align*}
k &\leq\left(\frac{sp+N(p-1)}{sp}\right)\left(  \frac{\left|f\right|_\infty}{\left|B_1\right|^{\frac{sp}{N}}\lambda_{1,p}^s(B_1)}\right) ^{\frac{N}{sp+N(p-1)}}[g(0)^{\frac{sp}{sp+N(p-1)}}-g(k)^{\frac{sp}{sp+N(p-1)}}]\\
&\leq\left(\frac{sp+N(p-1)}{sp}\right)\left(  \frac{\left|f\right|_\infty}{\left|B_1\right|^{\frac{sp}{N}}\lambda_{1,p}^s(B_1)}\right)  ^{\frac{N}{sp+N(p-1)}}\left(\left|u\right|_1\right)  ^{\frac{sp}{sp+N(p-1)}},
\end{align*}
since $g(k)\geq0$ and $g(0)=\left|u\right|_1$. 

Let $c$ denote, just for a moment, the right-hand side of the latter inequality. We have proved that $k\leq c$ whenever $\left|A_k\right|>0$.  Since $c$ does not depend on $k$ this implies that
$\left|A_k \right|=0$ for every $k>c$, thus allowing us to conclude that
$u\in L^\infty(\Omega)$ and also that $\left|u\right|_\infty\leq c$.  So,
\[\left|u\right|_\infty\leq\left(\frac{sp+N(p-1)}{sp}\right)\left(\frac{\left|f\right|_\infty}{\left|B_1 \right|^{\frac{sp}{N}}\lambda_{1,p}^s(B_1)}\right)^{\frac{N}{sp+N(p-1)}}\left|u\right|_1 ^{\frac{sp}{sp+N(p-1)}}%
\]
or, what is the same,%
\[\left|u\right|_\infty^{1+\frac{N(p-1)}{sp}}\leq\left(\frac{sp+N(p-1)}{sp}\right)^{\frac{sp+N(p-1)}{sp}}\left(\frac{\left|f\right|_\infty}{\left|B_1\right|^{\frac{sp}{N}}\lambda_{1,p}^s(B_1)}\right) ^{\frac{N}{sp}}\left|u\right|_1,
\]
from what follows \eqref{estimf}.
\end{proof}

\begin{corollary}
The $(s,p)$-torsion function $\phi_p^s$ belongs to $L^{\infty}(\Omega)$ and, in addition,
\begin{equation}
\frac{1}{\left|\Omega\right|}\leq\frac{\left|\phi_p^s  \right|_\infty}{\left|\phi_p^s\right|_1}\leq
\frac{1}{\left|B_1\right|}\left(  \frac{sp+N(p-1)}{sp}\right)
^{\frac{sp+N(p-1)}{sp}}\left(\frac{\lambda_{1,p}^s (\Omega)}{\lambda_{1,p}^s(B_1 )}\right)  ^{\frac{N}{sp}}. \label{ratio}%
\end{equation}
\end{corollary}
\begin{proof}
The first inequality is obvious. Proposition \ref{fbounds} with $u=\phi_p^s$ and $f\equiv1$ yields
\[
\frac{\left|\phi_p^s\right|_\infty}{\left|\phi_p^s  \right|_1}\leq\frac{1}{\left|B_1\right|}\left(\frac{sp+N(p-1)}{sp}\right)^{\frac{sp+N(p-1)}{sp}}\left(\frac{1}{\lambda_{1,p}^s(B_1)\left|\phi_p^s\right|_\infty^{p-1}}\right)^{\frac{N}{sp}}.
\]
Now, the second inequality in \eqref{ratio} follows from the first inequality in \eqref{eigenbounds}.
\end{proof}

\begin{theorem}
One has%
\[\lim_{p\rightarrow1^+}\frac{1}{\left|\phi_p^s\right|_\infty^{p-1}}=h_s(\Omega)=\lim_{p\rightarrow1^+}\frac{1}{\left|\phi_p^s\right|_1^{p-1}}.
\]
\end{theorem}
\begin{proof}
Taking \eqref{skaw} into account, we have
\[\lim_{p\rightarrow1^+}\frac{\lambda_{1,p}^s(\Omega)}{\lambda_{1,p}^s(B_1 )}=\frac{h_s(\Omega)}{h_s(B_1)}\in(0,\infty).
\]
Hence, it follows from \eqref{ratio} that
\[\lim_{p\rightarrow1^+}\left(\frac{\left|\phi_p^s\right|_\infty}{\left|\phi_p^s\right|_1}\right) ^{p-1}=1.
\]

Thus, by making $p$ go to $1$ in \eqref{eigenbounds} we have
\begin{align*}
\lim_{p\rightarrow1^+}\lambda_{1,p}^{s}(\Omega)& \leq\lim_{p\rightarrow1^+}\frac{\left|\Omega\right|^{p-1}}{\left|\phi_p^s\right|_1^{p-1}}\\
& =\lim_{p\rightarrow1^+}\frac{1}{\left|\phi_p^s\right|_1^{p-1}}\\
& =\lim_{p\rightarrow1^+}\left(\frac{\left|\phi_p^s\right|_\infty}{\left|\phi_p^s\right|_1}\right)^{p-1}\lim_{p\rightarrow1^+}\frac{1}{\left|\phi_p^s\right|_\infty^{p-1}}=\lim_{p\rightarrow1^+}\frac{1}{\left|\phi_p^s\right|_\infty^{p-1}}\leq\lim_{p\rightarrow1^+}\lambda_{1,p}^{s}(\Omega).
\end{align*}
Since $\lim_{p\rightarrow1^+}\lambda_{1,p}(\Omega)=h_s (\Omega)$, we are done.
\end{proof}\vspace*{.5cm}

In \cite{BLP}, the authors also proved that
\[h_s(\Omega)=\inf\left\{\left[v\right]_{s,1}^p   \,:\,v\in W_0^{s,1}(\Omega), \left|v\right|_1 =1\right\}.
\]
Since $W_0 ^{s,1}(\Omega)$ is not reflexive, they were able to prove that the minimum $h_s(\Omega)$ is attained on the larger Sobolev space
\[\mathcal{W}_0 ^{s,1}(\Omega):=\left\{v\in L^1  (\Omega)\,:\,\left[v\right]_{s,1}^p   <\infty,\ u=0\;\textrm{a.e. in }
\ \mathbb{R}^N\setminus\Omega\right\}.
\]

For completeness, we state the following result on the behavior of the $L^1$-normalized family $\left\{\frac{\phi_p^s}{\left|\phi_p^s\right|_1}\right\}$ as $p\rightarrow1$.  It corresponds to
\cite[Theorem 7.2]{BLP}, which was proved for the family $\left\{\frac{e_p^s}{\left|e_p^s\right|_1}\right\}  $.  Its proof
follows the same script and will be omitted.

\begin{theorem}
Let $u_p:=\frac{\phi_p^s}{\left|\phi_p^s\right|_1}$. 
There exists a sequence $\left(p_n\right)$ such that $p_n\rightarrow1^+$ and $u_{p_n}\rightarrow u$ in $L^q(\Omega)$, for every $q<\infty$.  The limit function $u$ is a solution of the minimization problem
\[h_s(\Omega)=\min_{v\in\mathcal{W}_0 ^{s,1}(\Omega)}\left\{\left[v\right]_{s,1}\,:\,u\geq0,\ \left|u\right|_1 =1\right\}  .
\]
Moreover, $u\in L^\infty(\Omega)$ and
\begin{equation}
\frac{1}{\left|\Omega\right|}\leq\left|u\right|_\infty\leq\frac{1}{\left|B_1\right|}\left(\frac{h_s(\Omega)}{h_s(B_1)}\right)^{\frac{N}{s}}.\label{uinf}%
\end{equation}
\end{theorem}

The upper bound that appears in the statement of Theorem 7.2 of \cite{BLP} is
\[\left[\frac{\left|B\right|^{\frac{N-s}{N}}}{P_s(B)}\right]^{\frac{N}{s}}h_s(\Omega)^{\frac{N}{s}}.
\]
However, it is very simple to check, by applying \eqref{FK2}, that it is equal to the upper bound in \eqref{uinf}.

We remark that, once obtained the convergence in $L^{q}(\Omega)$ stated above, the upper bound in \eqref{uinf} follows from \eqref{ratio}. Indeed,
since
\[
\left|u\right|_q=\lim_{n\rightarrow\infty}\left|u_{p_n}\right|_q\leq\left|\Omega\right|^{\frac{1}{q}}\lim_{n\rightarrow\infty}\left|u_{p_n}\right|_\infty
\]
\eqref{ratio} implies that
\[\left|\Omega\right|^{-\frac{1}{q}}\left|u\right|_q
\leq\lim_{n\rightarrow\infty}\left|u_{p_n}\right|_\infty\leq\frac{1}{\left|B_1\right|}\left(\frac{h_s (\Omega)}{h_s(B_1)}\right)^{\frac{N}{s}}.
\]
Hence, the upper bound in \eqref{uinf} follows, since $\left|u\right|_\infty=\lim_{q\rightarrow\infty}\left|\Omega\right|^{-\frac{1}{q}}\left|u\right|_q$. 

The lower bound in \eqref{uinf}, which does not appear in the statement of Theorem 7.2 of \cite{BLP}, follows by taking $q=1$, since
\[1=\lim_{n\rightarrow\infty}\left|u_{p_n}\right|_1 =\left|u\right|_1\leq\left|u\right|_\infty\left|\Omega\right|.
\]

It is interesting to note that, as it happens with the standard $p$-torsion functions, $\left|u\right|_\infty=\left|\Omega\right|^{-1}$ when $\Omega$ is a ball. In fact, in this case \eqref{FK2} yields
\[\frac{1}{\left|B_1\right|}\left(\frac{h_s(\Omega)}{h_s(B_1)}\right)^{\frac{N}{s}}=\frac{1}{\left|\Omega\right|}.
\]

\end{document}